\documentclass[a4paper]{amsart}
\usepackage{amsmath,amsthm,amssymb}
\newtheorem{theorem}{Theorem}[section]
\newtheorem{lemma}[theorem]{Lemma}
\newtheorem{proposition}{Proposition}[section]
\theoremstyle{definition}

\theoremstyle{remark}

\numberwithin{equation}{section}

\newcommand{\R}{\mathbb{R}}
\newcommand{\C}{\mathbb{C}}

\newcommand{\A}{\mathcal{A}}
\newcommand{\T}{\mathcal{T}}

\renewcommand{\L}{\mathcal{L}}

\renewcommand{\Re}{\operatorname{Re}}

\newcommand{\fra}{\mathfrak{a}}

\renewcommand{\mid}{\, \vert \,}

\begin{document}

\title{On evolution equations governed by non-autonomous forms}

\author{EL-Mennaoui Omar}
\address{Department of Mathematics, University Ibn Zohr, Faculty of Sciences, Agadir, Morocco}
\email{elmennaouiomar@yahoo.fr }

\author{Laasri Hafida}
\address{FernUniversität in Hagen – University of Hagen
Faculty of Mathematics and Computer Science,
Analysis, 58084 Hagen}
\email{hafida.laasri@fernuni-hagen.de}

\subjclass[2010]{35K90, 35K50, 35K45, 47D06.}

\keywords{Sesquilinear forms, non-autonomous evolution equations, maximal regularity, approximation.}

\begin{abstract}\label{abstract}
We consider a linear non-autonomous evolutionary Cauchy problem
\begin{equation}\label{eq00}
\dot{u} (t)+\A(t)u(t)=f(t) \hbox{ for }\  \hbox{a.e. t}\in [0,T],\quad u(0)=u_0,
\end{equation}
where the operator  $\A(t)$ arises from a time depending sesquilinear form $\fra(t,.,.)$ on a Hilbert space $H$ with constant domain $V.$
Recently a result on $L^2$-maximal regularity in $H,$ i.e., for each given $f\in L^2(0,T,H)$ and $u_0 \in V$ the problem  (\ref{eq00}) has a unique solution $u\in L^2(0,T,V)\cap H^1(0,T,H),$  is proved in \cite{D1} under the assumption
that $\fra$ is symmetric and of bounded variation.
The aim of this paper is to prove that the solutions of an approximate non-autonomous Cauchy problem in which $\fra$ is symmetric and  piecewise affine are closed to the solutions  of  that governed by symmetric and of bounded variation form. In particular, this provide an alternative proof of the result in \cite{D1} on $L^2$-maximal regularity in $H.$ 
\end{abstract}

\maketitle

\section{Introduction}\label{section:introduction}
In this work we are interested by  evolutionary linear equations of the form
\begin{equation}\label{eq0}
\dot{u} (t)+\A(t)u(t)=f(t), \quad u(0)=u_0,
\end{equation}
where the operators $\A(t), \ t\in[0,T]$ arise from time dependent
sesquilinear forms.  More precisely, let
 $H$ and $V$  denote  two separable Hilbert spaces such that $V$ is continuously and densely embedded into $H$ (we write $V \underset d \hookrightarrow H$). Let $V'$ be  the antidual of $V$ and denote by $\langle ., . \rangle$ the duality
between $V'$ and $V.$ As usual, we identify $H$ with  $H'$ and we obtain that  $V\underset d\hookrightarrow H\cong H'\underset d\hookrightarrow V'.$  These embeddings are continuous and dense (see e.g., \cite{Bre11}). Let
\[
    \fra: [0,T]\times V\times V \to \C
\]
be a\textit{ closed  non-autonomous form}, i.e., $\fra(t, .,.)$ is sesquilinear for all $t\in[0,T]$, $\fra(.,u,v)$ is measurable for all $u,v\in V,$
\begin{equation*}
    |\fra(t,u,v)| \le M \Vert u\Vert_V \Vert v\Vert_V \quad (t\in[0,T],u,v\in V)\qquad
\end{equation*}
and
\begin{equation*}
    \Re \fra (t,u,u) +\omega\Vert u\Vert_H^2\ge \alpha \|u\|^2_V \quad ( t\in [0,T], u\in V)
\end{equation*}
for some  $\alpha>0, M> 0$ and $\omega\in \R.$ The operator $\A(t) \in \L(V,V')$  associated with  $\fra(t,.,.)$  on $V'$ is defined for each $t\in[0,T]$ by
\[
\langle \A(t) u, v \rangle = \fra(t,u,v) \quad (u,v \in V).
\]

\par\noindent Seen as an unbounded operator on $V'$ with domain $D(\A(t)) = V,$ the
operator $- \A(t)$ generates a holomorphic $C_0-$semigroup $\T_t$ on $V'.$ The semigroup is bounded on a  sector if $\omega
=0$, in which case $\A$ is an isomorphism. We denote by $A(t)$ the part
of $\A(t)$ on $H$; i.e.,\
\begin{align*}
    D(A(t)) := {}& \{ u\in V : \A(t) u \in H \}\\
    A(t) u = {}& \A(t) u.
\end{align*}
It is a known fact that  $-A(t)$ generates a holomorphic $C_0$-semigroup  $T$ on $H$  and $T=\mathcal T_{\mid H}$ is the
restriction of the semigroup generated by $-\A$ to $H.$ Then $A(t)$ is the operator \textit{induced} by $\fra(t,.,.)$  on $H.$
 We refer to \cite{Ar06},\cite{Ka} and  \cite[Chap.\ 2]{Tan79}.\par\noindent
In 1961 J. L.  Lions proved that the non-autonomous Cauchy problem
\begin{equation}\label{nCP in V'}
\dot{u} (t)+\A(t)u(t)=f(t), \quad u(0)=u_0.
\end{equation}
 has \textit{$L^2$-maximal regularity in $V':$}
\begin{theorem}(\textrm{Lions})\label{wellposedness in V'}
 For all $f\in L^2(0,T;V^\prime)$ and $u_0\in H,$ the problem (\ref{nCP in V'})
has a unique solution $u \in \MR(V,V'):=L^2(0,T;V)\cap H^1(0,T;V').$
\end{theorem}
Lions proved this result in \cite{Lio61} (see also \cite[Chapter 3]{Tho03}) using a  representation theorem
of linear functionals due to himself and  usually known in the literature as  \textit{Lions's representation Theorem} and
  using  Galerkin's method in  \cite[ XVIII
Chapter 3, p. 620]{DL88}. We refer also to an alternative proof given by Tanabe \cite[Section 5.5]{Tan79}.

\par In Theorem \ref{wellposedness in V'} only measurability of $\fra: [0,T]\times V\times V \to \C$ with respect to the time variable is required to have a solution $u\in MR(V,V').$
 Nevertheless, in applications to boundary valued problems, like heat equations with non-autonomous Robin-boundary-conditions
 or Schr\"odinger equations with time-dependent potentials, this is not sufficient. One is more interested in
 \textit{$L^2$-maximal regularity in $H$} rather than in $V',$ i.e., in solutions which belong to \begin{equation}
MR(V,H):=L^2(0,T;V)\cap H^1(0,T;H)
\end{equation}
 rather than in  $MR(V,V').$ Lions  asked  a long time before  in \cite[p.\ 68]{Lio61} whether
 the solution $u$ of (\ref{nCP in V'}) belongs to  $MR(V,H)$ in the case where $\fra(t;u,v)=\overline{\fra(t;u,v)}$ and $t\mapsto \fra(t;u,v)$ is only measurable.
 
\par\noindent Dier \cite{D1} has recently showed that in general the unique assumption of measurability is not sufficient to have $u\in MR(V,H).$ However, several progress are already has been done by Lions 
\cite[p.~68, p.~94, ]{Lio61}, \cite[Theorem~1.1, p.~129]{Lio61} and
\cite[Theorem~5.1, p.~138]{Lio61} and also by Bardos \cite{Bar71}
under additional regularity assumptions on the form $\fra,$   the
initial value $u_0$ and the inhomogeneity $f.$ More recently, this
problem has been studied  with some progress and different
approachs by  Arendt, Dier, Laasri and Ouhabaz \cite{ADLO14}, Arendt
and Monniaux \cite{AM}, Ouhabaz \cite{O15}, Dier \cite{D2}, Haak and Ouhabaz
\cite{OH14}, Ouhabaz and Spina
\cite{OS10}. Results on multiplicative perturbations are also established in \cite{ADLO14,D2,AuJaLa14}. 
\par\noindent In \cite{LASA14} we proved Theorem \ref{wellposedness in V'} by a completely different
 approach developed in \cite{ELLA13} and \cite{LH}. The method uses an appropriate approximation of the   $\A(.).$  Namely, let
$\Lambda:=(0=\lambda_0<\lambda_1<...<\lambda_{n+1}=T)$ be a
subdivision of $[0,T].$ Consider the following approximation
$\A_\Lambda^{S}:\ [0,T]\rightarrow \mathcal{L}(V,V')$ of $\A$ given
by
\[ \A_\Lambda^{S}
(t):=\left\{%
\begin{array}{ll}
    \A_k & \hbox{for } \lambda_k\leq t<\lambda_{k+1},\\
    \A_{n} & \hbox{for } t=T, \\
\end{array}%
\right. \]  with

 \[\displaystyle \A_ku
:=\frac{1}{\lambda_{k+1}-\lambda_k}
\int_{\lambda_k}^{\lambda_{k+1}}\A(r)u{\rm  d}r\ \ \  (u\in V,
k=0,1,...,n).\] ("S" stands for step). The integral above makes sense since  $t\mapsto \mathcal A(t)u$ is Bochner integrable on $[0,T]$ with values in $V'$ for all $u\in V.$ Note that $\|\mathcal A(t)u\|_{V'}\leqslant M \|u\|_V$ for all $u\in V$ and all $t\in [0,T]$. It is worth to mention that the mapping $t\mapsto \mathcal A(t)$ is strongly measurable by the Dunford-Pettis Theorem \cite{ABHN11} since the spaces are assumed to be separable and $t\mapsto \mathcal A(t)$ is weakly measurable.
\par\noindent It has been proved in \cite[Theorem 3.2]{LASA14} that for all $u_0\in H$ and
$f\in L^2(0,T;V'),$  the non-autonomous problem
 \begin{equation}\label{nCP in V'0}
\dot{u}_\Lambda (t)+\A_\Lambda(t)u_\Lambda(t)=f(t), \quad u_\Lambda(0)=u_0
\end{equation}
 \noindent has an (explicit) unique solution $u_{\Lambda}\in MR(V,V'),$ and $(u_\Lambda)$  converges weakly in $MR(V,V')$ as $\vert
\Lambda\vert\rightarrow 0$ to the unique solution $u$ of  $(\ref{nCP in V'})$.
If we consider $u_0\in V$ and $f\in L^2(0,T; H)$ then the solution $u_\Lambda$ of  (\ref{nCP in V'0}) belongs to  $\MR(V,H)\cap C([0,T]; V)$ (see \cite{LH}, \cite{LASA14}). If moreover,  $\fra$ is assumed to be  piecewise
Lipschitz-continuous on $[0,T]$  then we obtain the convergence of $u_\Lambda$ in $\MR(V,H)$ \cite{LASA14} (see also
\cite{ADLO14}).\par\noindent In this paper we are concerned with the
recent result obtained in \cite{D2}. Instead of functions that are
constant on each subinterval $[\lambda_k,\lambda_{k+1}[$, we will
consider here those  that are linear in time.
\section{Preliminary}
Let $X$ be a Banach space and $T>0.$  Recall that a point $t\in
[0,T]$ is said to be a \textit{Lebesgue point} of a function  $f:\
[0,T]\longrightarrow X$ if \[\lim\limits_{h\to
0}\frac{1}{h}\int_t^{t+h}\|f(s)-f(t)\|_Xds=0.\] Clearly each point
of continuity of $f$ is a Lebesgue point. By \cite[Proposition
1.2.2]{ABHN11} if $f$ is Bochner integrable then almost all point
are Lebesgue points.
 \par\noindent Let $D$ be an other Banach space such that $D$ is continuously and densely embedded
 into $X$ and let $A: [0,T]\rightarrow \L(D,X)$ be a bounded and strongly measurable function, i.e, for each $x\in D$ the
  function $A(.)x:\ [0,T]\longrightarrow X.$ is measurable and
  bounded.
\par\noindent  Let
$\Lambda:=(0=\lambda_0<\lambda_1<...<\lambda_{n+1}=T)$ be a
subdivision of $[0,T].$ We consider the following approximations of
${A}:\ [0,T]\rightarrow \mathcal{L}(D,X)$ by step operator function
${A}_\Lambda^S:\ [0,T]\rightarrow \mathcal{L}(D,X)$ and piecewise linear
operator function ${A}_\Lambda^L:\ [0,T]\rightarrow
\mathcal{L}(D,X)$  given by \[{A}_\Lambda^S
(t):=\left\{%
\begin{array}{ll}
    A_k & \hbox{for } \lambda_k\leq t<\lambda_{k+1},\\
    A_{n} & \hbox{for } t=T, \\
\end{array}%
\right.\]

\noindent and
\begin{equation}
\begin{aligned} A_\Lambda^L(t):=\frac{\lambda_{k+1}-t}{\lambda_{k+1}-\lambda_k}A_{k}+\frac{t-\lambda_k}{\lambda_{k+1}-\lambda_k}A_{k+1},
\qquad \hbox { for } t\in [\lambda_k,\lambda_{k+1}],
\end{aligned}
\end{equation}

 \noindent where   \[\displaystyle A_kx
:=\frac{1}{\lambda_{k+1}-\lambda_k}
\int_{\lambda_k}^{\lambda_{k+1}}A(r)x{\rm  d}r\ \ \  (x\in D,
k=0,1,...,n).\] Let
$|\Lambda|:=\displaystyle\max_{j=0,1,...,n}(\lambda_{j+1}-\lambda_{j})$
denote the mesh of the subdivision $\Lambda.$ Assume that the subdivision $\Lambda$ is uniform, i.e., $\lambda_{k+1}-\lambda_{k}=T/n=\vert\Lambda\vert$ for all $k=0,1,...,n.$ In the following
Lemma, we show that $ A_\Lambda^S $ and  $ A_\Lambda^L $ converge
strongly and almost everywhere to $A$ as $\vert \Lambda\vert
\rightarrow 0,$ from which the strong convergence with respect to
$L^p$-norm ($p\in[1,\infty)$) follows.
{\lemma\label{lemma:Approximation classique-convergence des
operators} Let $A_{_\Lambda}^S: [0,T]\rightarrow \mathcal{L}(D,X)$
be given as above. Then:
\par $i)$ For all $x\in D$ we have $A_{_\Lambda}^S(t)x\rightarrow  A(t)x$  $t$-a.e. on
$[0,T]$ as $|\Lambda|\rightarrow 0.$
\par $ii)$ $A_{_\Lambda}^S(.)u_{_\Lambda}(.)\rightarrow A(.)u(.)$ in $L^p(0,T; X)$ as
$|\Lambda|\longrightarrow 0$ if $u_\Lambda\in L^p(0,T; D)$ such that $u_\Lambda\rightarrow u$ in $ L^p(0,T; D).$
}
\begin{proof} Let $C\geq 0$
such that $\Vert A(t)x\Vert_X\leq C\Vert x\Vert_D $ for all $x\in
D$ and for almost every  $t\in[0,T].$  We have $\Vert A_kx\Vert_X\leq C\Vert x\Vert_D$ for
all $x\in D$ and $k=0,1,...,n.$ Let $t$ be any Lebesgue point of $ A(.)x$. Let $k\in
\{0,1,...,n\}$ such that $t\in [\lambda_k,\lambda_{k+1})$. Then
\begin{align*}
&{A}_{_\Lambda}^S(t)x-A(t)x=\frac{1}{\lambda_{k+1}-\lambda_k}\int_{\lambda_k}^{\lambda_{k+1}}(A(r)x-A(t)x){\rm
d}r\\&
=\frac{1}{\lambda_{k+1}-\lambda_k}\int_{\lambda_k}^{t}(A(r)x-A(t)x)
{\rm d}r+
\frac{1}{\lambda_{k+1}-\lambda_k}\int_{t}^{\lambda_{k+1}}(A(r)x-A(t)x){\rm
d}r\\&\\&=(\frac{t-\lambda_k}{\lambda_{k+1}-\lambda_k})\frac{1}{t-\lambda_k}\int_{\lambda_k}^{t}(A(r)x-A
(t)x){\rm d}r\\&\ \ \ \ \ \ \ \ \ \ \ \ \ +
(\frac{\lambda_{k+1}-t}{\lambda_{k+1}-\lambda_k})\frac{1}{\lambda_{k+1}-t}\int_{t}^{\lambda_{k+1}}(A(r)x-A
 (t)x){\rm
d}r.\end{align*} \noindent It follows that $ {A}_{_\Lambda}^S(t)x-A(t)x\rightarrow 0$ as
$\vert \Lambda\vert \rightarrow 0.$  Since
almost all points of $[0,T]$ are  Lebesgue points of $A(.)x$  the first assertion follows
\par\noindent For the second assetion let $x\in D$ and let
$\Omega$ be a measurable subset of $[0, T].$ We set $w=x\otimes
1_\Omega.$ Then  $\Vert A_{_\Lambda}^S w-{ A}w\Vert_{L^p(0,T,X)}^p=
\int_\Omega\Vert A_{_\Lambda}^S(t)x-A(t)x\Vert_{X}^p{\rm
d}t\rightarrow 0$ as $|\Lambda|\longrightarrow 0$ by $i)$ and
Lebesgue's Theorem. From which follows that $\Vert {
A}_{_\Lambda}^Sw-{ A}w\Vert_{L^p(0,T;X)}\rightarrow 0$ as
$|\Lambda|\longrightarrow 0$ for all simple function $w$ and thus
for all $w\in L^p(0,T;D).$ Let now $w_\Lambda\in L^p(0,T; D)$
such that $w_\Lambda\rightarrow w$ in $ L^p(0,T;D).$ Then \[
\|{A}_{\Lambda}^Sw_{\Lambda}-{A}w\|_{L^p(0,T;X)}\leq
C\|w_{\Lambda}-w\|_{L^p(0,T;D)}+\|{A}_{\Lambda}^Sw-{A}w\|_{L^p(0,T;X)}.\]
Thus $(ii)$  holds.
\end{proof}
 Instead of  functions that are constant on each subinterval $[\lambda_k,\lambda_{k+1}[$, we consider now those  that are linear.
 \begin{lemma}\label{Proposition: Approximation Lipschitz}
 Let $A:[0,T]\longrightarrow \L(D,X)$ be a bounded and strongly measurable function. Then the following statements hold:
\medskip

$1.$ For all $x\in D$ we have $ A_{_\Lambda}^L(t)x\rightarrow  A(t)x$  $t$-a.e. on $[0,T]$ as $|\Lambda|\rightarrow 0.$
\medskip

$2.$ $A_{_\Lambda}^L(.)u_{_\Lambda}(.)\rightarrow A(.)u(.)$ in $L^p(0,T; X)$ as
$|\Lambda|\longrightarrow 0$ if $u_\Lambda\in L^p(0,T; D)$ such that $u_\Lambda\rightarrow u$ in $ L^p(0,T; D).$

\end{lemma}
\begin{proof} Let $x\in D$ and let $t\in[0,T]$ be an arbitrary Lebesgue point of $A(.)x$ and $k\in \{0,1,...,n\}$ be such that $t\in [\lambda_k,\lambda_{k+1}].$ Then
\begin{align*}A_{_\Lambda}^L(t)x-A(t)x&=
\frac{\lambda_{k+1}-t}{\lambda_{k+1}-\lambda_k}\big(A_{k}x-A(t)x\big)+\frac{t-\lambda_k}{\lambda_{k+1}-\lambda_k}\big(A_{k+1}x-A(t)x\big)
\\&=I+II
\end{align*}
For the first term $I$  we have \[I=\frac{\lambda_{k+1}-t}{\lambda_{k+1}-\lambda_k}\big({A}_{_\Lambda}^S(t)x-A(t)x\big)\]
which converges  to zero as $\vert\Lambda\vert\longrightarrow 0$ by Lemma \ref{lemma:Approximation classique-convergence des operators}.
Now we show that $II$ converges also to zero as $\vert\Lambda\vert$ goes to $0$. Indeed, we have
\begin{align}\nonumber A_{k+1}x-A(t)x&=\frac{1}{\lambda_{k+2}-\lambda_{k+1}}\int_{t}^{\lambda_{k+2}}(A(r)-A(t))xdr\\&\nonumber\qquad\qquad-
\frac{1}{\lambda_{k+2}-\lambda_{k+1}}\int_{t}^{\lambda_{k+1}}(A(r)-A(t))xdr
\\\label{eq-proof-affine approximation-operator1}&=(\frac{\lambda_{k+2}-t}{\lambda_{k+2}-\lambda_{k+1}})\frac{1}{\lambda_{k+2}-t}\int_{t}^{\lambda_{k+2}}(A(r)-A(t))xdr
\\\label{eq-proof-affine approximation-operator2}&\qquad\qquad-(\frac{\lambda_{k+1}-t}{\lambda_{k+2}-\lambda_{k+1}})\frac{1}{\lambda_{k+1}-t}\int_{t}^{\lambda_{k+1}}(A(r)-A(t))xdr
\end{align}
Using again \cite[Proposition 1.2.2]{ABHN11} we obtain that both terms in $(\ref{eq-proof-affine approximation-operator1})$ and $(\ref{eq-proof-affine approximation-operator2})$ converges to $0$ as $\vert\Lambda\vert\longrightarrow 0.$ Consequently $II$ converges to $0.$ The claim follows since $t$ is arbitrary Lebesgue point of $A(.)x.$ The proof of $(2)$ is the same as the proof  of $(ii)$ in Lemma \ref{lemma:Approximation classique-convergence des operators}.
\end{proof}
\section{Approximation and convergence}


In this section  $H,V$ are  complex  separable Hilbert spaces such that $V \underset d \hookrightarrow H.$ Let $T>0$ and let
\[
    \fra: [0,T]\times V\times V \to \C
\]
be a non-autonomous closed form. This  means that $\fra(t, .,.)$ is sesquilinear for all $t\in[0,T]$, $\fra(.,u,v)$ is measurable for all $u,v\in V,$
\begin{equation}\label{eq:continuity-nonaut}
    |\fra(t,u,v)| \le M \Vert u\Vert_V \Vert v\Vert_V \quad (t\in[0,T],u,v\in V)\qquad
\end{equation}
and
\begin{equation}\label{eq:Ellipticity-nonaut}
    \Re \fra (t,u,u) +\omega\Vert u\Vert\ge \alpha \|u\|^2_V \quad ( t\in [0,T], u\in V)
\end{equation}
for some  $\alpha>0, M\geq 0$ and $\omega\in \R.$ We assume in addition that $\fra$ is \emph{symmetric}; i.e.,
\begin{equation*}\label{symmetric}
    \fra(t,u,v) = \overline{\fra(t,v,u)} \quad (t \in [0,T], u,v \in V).
\end{equation*}
\noindent  For almost every $t\in[0,T]$ we denote by $\A(t)\in \L(V,V')$  the operator associated with the form $\fra(t,.,.)$ in $V'.$
 The non-autonomous Cauchy problem (\ref{nCP in V'}) has $L^2$-\textit{maximal regularity in} $V',$ i.e., for given $f\in L^2(0,T;V')$ and $u_0\in H,$ (\ref{nCP in V'}) has a unique solution $u$ in $MR(V,V')=L^2(0,T;V)\cap H^1(0,T;V').$ The maximal regularity space
$\MR (V,V')$ is continuously embedded into $C([0,T],H)$ and  if  $u\in \MR (V,V')$ then
the function  $\|u(.)\|^2$ is absolutely continuous on $[0,T]$ and
\begin{equation}\label{poduct rule classique}
\frac{d}{dt}\Vert u(.)\Vert^2_H=2\Re\langle\dot u(.),u(.)\rangle
\end{equation}
see e.g., \cite[Chapter III, Proposition 1.2]{Sho97} or \cite[Lemma 5.5.1]{Tan79}.
\par\noindent For simplicity  we may assume without loss of generality that  $\omega=0$ in (\ref{eq:Ellipticity-nonaut}). In fact, let $u\in MR(V,V')$ and let $v:=e^{-w.}u.$ Then $v\in MR(V,V')$ and it satisfies
\begin{equation*}
  \dot{v}(t)+(\omega+\mathcal A(t))v(t)=e^{-wt}f(t) \ \ \  t{\rm
-a.e.}  \hbox{ on} \ [0,T],\
 \ \ \ \ v(0)=0 \
\end{equation*}
if and only if $u$ satisfies (\ref{nCP in V'}).
Throughout this section $\omega=0$ will be our assumption.
\par\noindent Let $\Lambda=(0=\lambda_0<\lambda_1<...<\lambda_{n+1}=T)$ be
a uniform subdivision of $[0,T].$ Let \[\fra_k:V \times V \to \C\ \ \hbox{ for } k=0,1,...,n\] be the family of sesquilinear  forms given for all $u,v\in V$ and $k=0,1,...,n$ by
\begin{equation}\label{eq:form-moyen integrale}
 \fra_k(u,v):=\frac{1}{\lambda_{k+1}-\lambda_k}
\int_{\lambda_k}^{\lambda_{k+1}}\fra(r;u,v){\rm  d}r.\ \ \
\end{equation}
Remark that $\fra_k$ satisfies (\ref{eq:continuity-nonaut}) and (\ref{eq:Ellipticity-nonaut}) for all $k=0,1,...n.$ The associated operators are denoted by $\A_k\in \L(V,V')$ and are given for all $u\in V$ and $k=0,1,...,n$ by
\begin{equation}\label{eq:op-moyen integrale}
 \A_ku :=\frac{1}{\lambda_{k+1}-\lambda_k}
\int_{\lambda_k}^{\lambda_{k+1}}\A(r)u{\rm  d}r.\ \ \  \end{equation}
 This  integral  is well defined. Indeed, the mapping $t\mapsto \mathcal A(t)$ is strongly measurable by the Pettis Theorem \cite{ABHN11} since  $t\mapsto \mathcal A(t)$ is weakly measurable and the spaces are assumed to be separable. On the other hand, $\|\mathcal A(t)u\|_{V'}\leqslant M \|u\|_V$ for all $u\in V$ and a.e. $t\in [0,T].$ Thus $t\mapsto \mathcal A(t)u$ is Bochner integrable on $[0,T]$ with values in $V'$ for all $u\in V.$
\par\noindent The function
\begin{equation}\label{form: approximation formula1}\fra_\Lambda^L:[0,T]\times V \times V \to \C\end{equation}
defined for $t\in [\lambda_k,\lambda_{k+1}]$ by
\begin{equation}\label{form: approximation formula2}
 \fra_{\Lambda}^L(t;u,v):=\frac{\lambda_{k+1}-t}{\lambda_{k+1}-\lambda_k}\fra_{k}(u,v)+\frac{t-\lambda_k}{\lambda_{k+1}-\lambda_k}\fra_{k+1}(u,v) \ \ (u,v\in V),
\end{equation}
 is a symmetric non-autonomous closed form and Lipschitz continuous with respect to the time variable $t\in [0,T].$ The associated time dependent operator is denoted by
  \begin{equation}\label{Operator: approximation formula1}\A_\Lambda^L(.):[0,T]\to \L(V,V')
\end{equation}
 and is given by
\begin{equation}\label{Operator: approximation formula2}\
 \A_{\Lambda}^L(t):=\frac{\lambda_{k+1}-t}{\lambda_{k+1}-\lambda_k}\A_{k}+\frac{t-\lambda_k}{\lambda_{k+1}-\lambda_k}\A_{k+1} \hbox{   for } t\in [\lambda_k,\lambda_{k+1}]
\end{equation}
Since $\fra_k, k=0,1,...,n$ are symmetric, 
 the function $\fra_k(v(.))$ belongs to $W^{1,1}(a,b)$ and the following rule formula 
\begin{equation}\label{rule-formula}
\dot\fra_k(v(t)):=\frac{d}{dt}\fra_k(v(t))=2(A_kv(t)\mid \dot v(t)) \hbox{  for a.e. } t\in [a,b],
\end{equation}
holds  whenever $v\in H^1(a,b,H)\cap L^2(a,b,D(A_k)),$ for all $[a,b], k=0,1,...,n$   where $A_k$ is the part of $\A_k$ in $H.$ For the proof  we refer to \cite[Lemma 3.1]{AC10}.

\begin{theorem}\label{ADLO14} Given $f\in L^2(0,T;H)$ and $u_0\in V,$ there is a unique solution $u_\Lambda \in \MR(V,H)$ of
\begin{equation}\label{Probleme approche:  H}
\dot{u}_\Lambda (t)+\A_\Lambda^L(t)u_\Lambda(t)=f(t), \quad u_\Lambda(0)=u_0.
\end{equation}
Moreover, $t\mapsto \fra_\Lambda(t,u_\Lambda(t))\in W^{1,2}(0,T)$ and
\begin{equation}\label{product rule22}
2\Re(\A_\Lambda^L(t)u_\Lambda(t)\mid \dot{u}_\Lambda(t))_H=\frac{d}{dt}(\fra_\Lambda^L(t;u_\Lambda(t))
-\dot{\fra}_\Lambda^L(t;u_\Lambda(t))\qquad \ t.\hbox{a.e}
\end{equation}
\end{theorem}
\begin{proof}
 The first part of the theorem follows from \cite{Lio61}, \cite[Theorem 4.2]{ADLO14},\cite{LASA14} since
 $t\mapsto \fra_\Lambda^L(t,u,v)$ is piecewise $C^{1}$ for all $u,v\in V.$ The rule product  \label{product rule}
   follows also from \cite[Theorem 3.2]{ADLO14}, but it can be also seen directly from
   \begin{align*} \fra_\Lambda^L(t;u_\Lambda(t))&=\int_0^t2\Re(\A_\Lambda(s)u_\Lambda(s)|\dot u(s))_Hds
    \\&+\int_0^t\dot\fra_\Lambda^L(r,u_\Lambda(r))dr+\fra_\Lambda^L(0,u_0) \ (t\in [0,T])
    \end{align*}
 which holds for all $t\in [0,T].$ In fact, let $\delta>0,$ $t\in[0,T]$ be arbitrary and let $l\in \{0,1,...,n\}$ be such that $t\in[\lambda_l,\lambda_{l+1}].$ 
 In order to apply the classical product rule  (\ref{rule-formula}), we seek regularizing $u_\Lambda$  by multiplying  with $e^{-\delta A_k}$ and $e^{-\delta A_{k+1}}.$ Then
  \begin{align*}
  \int_{\lambda_k}^{\lambda_{k+1}}&(\A_\Lambda(s)u_\Lambda(s)|\dot u_\Lambda(s))_H ds
     \\&=\lim\limits_{\delta\to 0}
\int_{\lambda_k}^{\lambda_{k+1}}\Big(\frac{\lambda_{k+1}-r}{\lambda_{k+1}-\lambda_k}(\A_k e^{-\delta A_k}u_\Lambda(s)|\dot u_\Lambda(s))_H
 \\&\qquad\qquad\qquad\qquad+ \frac{r-\lambda_{k}}{\lambda_{k+1}-\lambda_k}(\A_{k+1} e^{-\delta A_{k+1}}u_\Lambda(s)|\dot u_\Lambda(s))_H
  \Big)ds
	\end{align*}
for $k=0,1,...,l-1.$	
Using (\ref{rule-formula}) and integrating by part we obtain by an easy  calculation
\begin{align*}
2&\Re\int_{\lambda_k}^{\lambda_{k+1}}(\A_\Lambda(s)u_\Lambda(s)|\dot u_\Lambda(s))_H ds
\\&=\lim\limits_{\delta\to 0}\Big[\fra_{k+1}(e^{-\frac{\delta}{2}A_{k+1}}u_\Lambda(\lambda_{k+1}))-\fra_{k}(e^{-\frac{\delta}{2}A_{k}}u_\Lambda(\lambda_{k}))\Big]
\\&\qquad-\lim\limits_{\delta\to 0}\int_{\lambda_k}^{\lambda_{k+1}}\frac{1}{\lambda_{k+1}-\lambda_k}\Big[\fra_{k+1}(e^{-\frac{\delta}{2}A_{k+1}}u_\Lambda(s))-\fra_{k}(e^{-\frac{\delta}{2}A_{k}}u_\Lambda(s))\Big]ds
\\&=\fra_{k+1}(u_\Lambda(\lambda_{k+1}))-\fra_{k}(u_\Lambda(\lambda_{k}))
-\int_{\lambda_k}^{\lambda_{k+1}}\frac{1}{\lambda_{k+1}-\lambda_k}\Big[\fra_{k+1}(u_\Lambda(s))-\fra_{k}(u_\Lambda(s))\Big]ds
\\&=\fra_{k+1}(u_\Lambda(\lambda_{k+1}))-\fra_{k}(u_\Lambda(\lambda_{k}))
-\int_{\lambda_k}^{\lambda_{k+1}}\dot\fra_{\Lambda}(s,u_\Lambda(s))ds
\end{align*}
for $k=0,2,...,l-1,$	here we have use that the restriction of  $(e^{-tA_k})_{t\geq 0}$ on $V$ is a $C_0$-semigroup.   By a similar argument as above we obtain for the integral over $(\lambda_l,t)$

  \begin{align*}&2\Re \int_{\lambda_l}^t(\A_\Lambda(s)u_\Lambda(s)|\dot u_\Lambda(s))_H ds
	\\&=\frac{\lambda_{l+1}-t}{\lambda_{l+1}-\lambda_{l}}\fra_{l}(u_\Lambda(t))+
\frac{t-\lambda_{l}}{\lambda_{l+1}-\lambda_{l}}\fra_{l+1}(u_\Lambda(t))
-\fra_{l}(u_\Lambda(\lambda_l))
\\&-\int_{\lambda_l}^{t}\frac{1}{\lambda_{l+1}-\lambda_l}\Big[\fra_{l+1}(u_\Lambda(s))-\fra_{l}(u_\Lambda(s))\Big]ds
\\&=\fra_\Lambda^L(t,u_\Lambda(t))-\fra_{l}(u_\Lambda(\lambda_l))-
\int_{\lambda_l}^{t}\dot\fra_{\Lambda}(s,u_\Lambda(s))ds
\end{align*}
Consequently

\begin{align*}
  2&\Re\int_0^t(\A_\Lambda(s)u_\Lambda(s)|\dot u_\Lambda(s))_Hds
 \\&=2\Re\sum_{k=0}^{l-1}\int_{\lambda_k}^{\lambda_{k+1}}(\A_\Lambda(s)u_\Lambda(s)|\dot u_\Lambda(s))_H ds+
    2\Re\int_{\lambda_l}^t(\A_\Lambda(s)u_\Lambda(s)|\dot u_\Lambda(s))_H ds
\\&=-\fra_0(u_0)+\fra_\Lambda^L(t,u_\Lambda(t))-\int_0^{t}\dot\fra_\Lambda^L(r,u_\Lambda(r))dr
\end{align*}
This completes the  proof.
   \end{proof}
 The next proposition shows that
 $u_\Lambda$ from Theorem \ref{ADLO14} approximates the solution of (\ref{nCP in V'}) with respect to the norm of $MR(V,V').$
\begin{proposition}\label{convergence forte de solution approchee} Let $f\in L^2(0,T;H)$ and $u_0\in V$ and let $u_\Lambda\in MR(V,H)$ be the solution of (\ref{Probleme approche:  H}). Then $u_\Lambda$ converges strongly in $MR(V,V')$ as $|\Lambda|
  \longrightarrow 0$ to  the solution
of (\ref{nCP in V'}).
\end{proposition}
\begin{proof} Let $f\in L^2(0,T;H)$ and $u_0\in V.$ Let $u, u_\Lambda\in MR(V,V')$ be the solution of (\ref{nCP in V'}) and (\ref{Probleme approche:  H}) respectively. 
Set $w_\Lambda:=u_\Lambda-u$ and $g_{\Lambda}:=(\A-\A_{\Lambda}^L)u.$  Then  $w_\Lambda\in MR(V,V')$ and satisfies
\begin{equation*}
\dot{w}_\Lambda (t)+\A_\Lambda^L(t)w_\Lambda(t)=g_{\Lambda}(t), \quad w_\Lambda(0)=0.
\end{equation*}
From the product rule (\ref{poduct rule classique})  it follows
\begin{align*}\frac{d}{dt}\|w_\Lambda(t)\|^2_H&=2 \Re \langle g_{_\Lambda}(t)-\A_\Lambda^L(t)w_{_\Lambda}(t),w_{_\Lambda}(t)\rangle
\\&=-2\Re a_\Lambda^L(t,w_{_\Lambda}(t),w_{_\Lambda}(t))+2 \Re \langle g_{_\Lambda}(t),w_{_\Lambda}(t)\rangle
\end{align*}
for almost every $t\in [0,T].$
Integrating this equality  on $(0,t),$  we obtain
\[\alpha\int_0^t\Vert w_{_\Lambda}(s)\Vert_V^2ds\leq \int_0^t
\Vert g_{_\Lambda}(s)\Vert_V'\Vert w_{_\Lambda}(s)\Vert_V ds.\]
 This estimate  and the Young's inequality  \[ab\leq \frac{1}{2}(\frac {a^2}{\varepsilon}+\varepsilon b^2) \ \ (\varepsilon>0,\ a,b\in \mathbb{{R}}).\]
yield the estimate
\begin{equation*}
 \alpha\Vert w_\Lambda\Vert_{L^2(0,T:V)}^2\leq 1/\alpha\Vert g_{\Lambda}\Vert_{L^2(0,T:V')}^2.
\end{equation*}
The term of the right hand side of this inequality converges by Proposition \ref{Proposition: Approximation Lipschitz} to $0$ as $|\Lambda|
  \longrightarrow 0.$ It follows that $u_\Lambda\longrightarrow u$ strongly in $L^2(0,T;V).$ Again from the second assertion of Proposition \ref{Proposition: Approximation Lipschitz} follows that $\A_{_\Lambda}^L u_\Lambda\longrightarrow \A u$  in $L^2(0,T;V').$ Letting $|\Lambda|$ go to
  $0$ in  \[\dot w_{_\Lambda}=\dot u_{_\Lambda}-\dot u=f-\A_{_\Lambda}^L u_\Lambda -\dot u\]
and recalling the continuous embedding of $MR(V,V')$ into $C([0,T];H)$ imply the claim.
\end{proof}
Next we assume additionally, as in \cite{D1} or \cite{D2}, that there exists a bounded and non-decreasing function $g:[0,T]\longrightarrow \L(H)$ such that
\begin{equation}\label{Variation bornee}
\vert \fra(t;u,v)-\fra(s;u,v)\vert\leq (g(t)-g(s))\Vert u\Vert_V\Vert v\Vert_V
\end{equation}
for  $u,v\in V, s,t\in [0,T], s\geq t.$ Our aim is the show that under this assumption the solution  $u_\Lambda$ of (\ref{Probleme approche:  H}) converges weakly in $MR(V,H)$ as $\vert\Lambda\vert\longrightarrow 0$ and that  the limit satisfies  (\ref{nCP in V'}).
Without loss of generality, we will assume that $g(0)=0.$ Thus $g$ is positive.  Let \[g_\Lambda^L:[0,T]\longrightarrow [0,\infty[\]  denote  the analogous function to (\ref{Operator: approximation formula1}) and (\ref{Operator: approximation formula2}) for $g.$ Assume that the subdivision $\Lambda$ is uniform, i.e., $\lambda_{k+1}-\lambda_{k}=T/n=\vert\Lambda\vert$ for all $k=0,1,...,n.$
\begin{lemma}\label{Lemma: variation bornee approche} 
\begin{equation}\label{Variation borne: form approchee} 
\vert \fra_\Lambda^L(t;u,v)-\fra_\Lambda^L(s;u,v)\vert\leq [g_\Lambda^L(t)-g_\Lambda^L(s)]\Vert u\Vert_V\Vert v\Vert_V
\end{equation}
for all  $u,v\in V$ and $t,s\in [0,T]$ with $s\leq t.$ 
\end{lemma}
\begin{proof}
It suffices to show (\ref{Variation borne: form approchee}) for $t,s\in[\lambda_{k},\lambda_{k+1}]$ for some $k\in \{0,1,...,n\}.$ The general case where $t,s$ belong to  two different subintervals follows immediately. Let $u,v\in V,$ then
\begin{align*}
\fra_\Lambda^L(t;u,v)-&\fra_\Lambda^L(s;u,v)=\frac{t-s}{\lambda_{k+1}-\lambda_{k}}\fra_{k+1}(u,v)-
\frac{t-s}{\lambda_{k+1}-\lambda_{k}}\fra_{k}(u,v)
\\&=\frac{t-s}{\lambda_{k+1}-\lambda_{k}}\frac{n}{T}\int_0^{T/n}[\fra(r+\lambda_{k+1};u,v)-\fra(r+\lambda_{k};u,v)]dr.
\end{align*}
Thus (\ref{Variation bornee}) implies
\begin{align*}
\vert \fra_\Lambda^L(t;u,v)-&\fra_\Lambda^L(s;u,v)\vert\\&\leq \frac{t-s}{\lambda_{k+1}-\lambda_{k}}\frac{n}{T}\int_0^{T/n}[g(r+\lambda_{k+1})-g(r+\lambda_{k})]dr\Vert u\Vert_V\Vert v\Vert_V
\\&=\frac{t-s}{\lambda_{k+1}-\lambda_{k}}\frac{n}{T}\big[\int_{\lambda_{k+1}}^{\lambda_{k+2}}g(r)dr
-\int_{\lambda_{k}}^{\lambda_{k+1}}g(r)dr\big]\Vert u\Vert_V\Vert v\Vert_V
\\&=\frac{t-s}{\lambda_{k+1}-\lambda_{k}}\big[\frac{1}{\lambda_{k+2}-\lambda_{k+1}}\int_{\lambda_{k+1}}^{\lambda_{k+2}}g(r)dr
\\&\qquad\qquad-\frac{1}{\lambda_{k+1}-\lambda_{k}}\int_{\lambda_{k}}^{\lambda_{k+1}}g(r)dr\big]\Vert u\Vert_V\Vert v\Vert_V
\\&=\big[\frac{t-s}{\lambda_{k+1}-\lambda_{k}}g_{k+1}-
\frac{t-s}{\lambda_{k+1}-\lambda_{k}}g_{k}\big]\Vert u\Vert_V\Vert v\Vert_V
\\=[g_\Lambda^L(t)-g_\Lambda^L(s)]\Vert u\Vert_V\Vert v\Vert_V
\end{align*}
\end{proof}
The main result of this section is the following
\begin{theorem}
    Assume that the non-autonomous closed  form $\fra$ is symmetric and satisfies (\ref{Variation bornee}). Let $f\in L^2(0,T;H)$ and $u_0\in V$ and let $u_\Lambda\in MR(V,H)$ be the
    solution of (\ref{Probleme approche:  H}). Then $(u_\Lambda)$ converges weakly in $MR(V,H)$ as $|\Lambda|
    \longrightarrow 0$ and $u=\lim\limits_{|\Lambda|\to 0}u_\Lambda$ satisfies  (\ref{nCP in V'}).
\end{theorem}
\begin{proof}
a) First since $u_\Lambda$ satisfies (\ref{Probleme approche:  H}) then
\[\Vert\dot u_{\Lambda}(t)\Vert_H
+(A_{_\Lambda}^L(t) u_{\Lambda}(t)\mid\dot u_{\Lambda}(t))_H=(f(t)\mid\dot u_{\Lambda}(t))_H\ \ t.a.e\]
The product rule (\ref{product rule22}), Cauchy-Schwartz inequality and
  Young's inequality imply that for almost every
$t\in[0,T]$
\[\|\dot{u}_{\Lambda}(t)\|^2_H+\frac{d}{dt}(\fra_\Lambda^L(t;u_\Lambda(t)))
\leq\|f(t)\|^2_H+\dot\fra_\Lambda^L(t;u_\Lambda(t)).\]
Integrating now this inequality on $[0,t],$ it follows that
\begin{equation}\label{equation of part a)}
  \int_0^t\Vert \dot u_\Lambda(r)\Vert^2_{H}{\rm d}r+\alpha\Vert u_\Lambda(t)\Vert_V^2\leq M\Vert  u_0\Vert_V^2
                 +\int_0^t\Vert  f(r)\Vert^2_{H}{\rm d}r+\int_0^t\dot\fra_\Lambda^L(r;u_\Lambda(r)){\rm d}r
\end{equation}
where $\alpha$ and $M$ are the constants in (\ref{eq:continuity-nonaut})-(\ref{eq:Ellipticity-nonaut}).
\par b) Note that by construction the derivative $\dot \fra_\Lambda^L$ of $\fra_\Lambda^L$ equals  \[\dot \fra_\Lambda^L(r;u)=\frac{\fra_{k+1}(u)-\fra_{k}(u)}{\lambda_{k+1}-\lambda_{k}}  \quad\hbox{ for a.e.} r\in[\lambda_{k},\lambda_{k+1}],\  u\in V.\] Now, let $t\in[0,T]$ be arbitrary and let $l\in \{0,1,...,n\}$ be such that $t\in[\lambda_l,\lambda_{l+1}].$ Then
\begin{align*}
    \int_0^t \dot \fra_\Lambda^L(r;u_\Lambda(r)){\rm d}r&=
    \sum_{k=1}^{l-1}\int_{\lambda_k}^{\lambda_{k+1}}\dot \fra_\Lambda^L(r;u_\Lambda(r)){\rm d}r+
    \int_{\lambda_l}^t\dot \fra_\Lambda^L(r;u_\Lambda(r)){\rm d}r
\\&=\sum_{k=1}^{l-1}\int_{\lambda_k}^{\lambda_{k+1}}\frac{\fra_{k+1}(u_\Lambda(r))-\fra_{k+1}(u_\Lambda(r))}{\lambda_{k+1}-\lambda_{k}}{\rm d}r
\\&\qquad\qquad+\int_{\lambda_l}^t\frac{\fra_{l+1}(u_\Lambda(r))-\fra_{l}(u_\Lambda(r))}{\lambda_{l+1}-\lambda_{l}}{\rm d}r
\\&=\sum_{k=1}^{l-1}\int_{\lambda_k}^{\lambda_{k+1}}\frac{\fra_{\Lambda}^L(\lambda_{k+1};u_\Lambda(r)) -\fra_{\Lambda}^L(\lambda_{k};u_\Lambda(r))}{\lambda_{k+1}-\lambda_{k}}{\rm d}r
\\&\qquad\qquad+\int_{\lambda_l}^t\frac{\fra_{\Lambda}^L(\lambda_{l+1};u_\Lambda(r)) -\fra_{\Lambda}^L(\lambda_{l};u_\Lambda(r))}{\lambda_{l+1}-\lambda_{l}}{\rm d}r.
\end{align*}
By Lemma \ref{Lemma: variation bornee approche} it follows that
\begin{align*}
\int_0^t \dot \fra_\Lambda^L(r;u_\Lambda(r)){\rm d}r&\leq \sum_{k=1}^{l-1}\int_{\lambda_k}^{\lambda_{k+1}}\frac{g_{\Lambda}^L(\lambda_{k+1}) -g_{\Lambda}^L(\lambda_{k})}{\lambda_{k+1}-\lambda_{k}}\Vert u_\Lambda(r)\Vert_V^2{\rm d}r
\\&\qquad\qquad+\int_{\lambda_l}^t\frac{g_{\Lambda}^L(\lambda_{l+1}) -g_{\Lambda}^L(\lambda_{l})}{\lambda_{l+1}-\lambda_{l}}\Vert
 u_\Lambda(r)\Vert_V^2{\rm d}r
\\&=\sum_{k=1}^{l-1}\int_{\lambda_k}^{\lambda_{k+1}}\dot g_{\Lambda}^L(r)\Vert u_\Lambda(r)\Vert_V^2{\rm d}r
+\int_{\lambda_l}^t \dot g_{\Lambda}^L(r) \Vert
 u_\Lambda(r)\Vert_V^2{\rm d}r
\\&=\int_{0}^t \dot g_{\Lambda}^L(r) \Vert
 u_\Lambda(r)\Vert_V^2{\rm d}r
\end{align*}
\par\noindent c) Using an analogous calculus as in part  $b)$ and the fact that \[\dot g_{\Lambda}^L(r)=\frac{g_{k+1}-g_{k}}{\lambda_{k+1}-\lambda_{k}}  \quad\hbox{ for a.e.} r\in[\lambda_{k},\lambda_{k+1}]\]
we can easily see that
\begin{equation}
\int_{0}^t \dot g_{\Lambda}^L(r){\rm d}r\leq g(T)
\end{equation}
since the function $g$ is positive and non-decreasing.
\par\noindent d) As a consequence of (\ref{equation of part a)}), the parts b)-c) and Gronwall's lemma it follows that
\begin{equation*}
\sup_{t\in[0,T]}\Vert u_\Lambda(t)\Vert_V^2\leq 1/\alpha\big[M\Vert  u_0\Vert_V^2
                 +\int_0^T\Vert f(r)\Vert^2_{H}{\rm d}r\big]\exp(g(T)/\alpha).
\end{equation*}
Inserting this estimate into (\ref{equation of part a)}), we find that there exists $c=c(\alpha, g(T), M)\geq 0$ such that
\begin{equation}\label{estimation de la derivee approche}
\int_0^T \Vert\dot u_\Lambda(s)\Vert_{H}^2{\rm d}s\leq c\big[\Vert  u_0\Vert_V^2+\Vert f\Vert_{L^2(0,T;H)}^2\big]
\end{equation}
Since $u_\Lambda(t)=u_\Lambda(0)+\int_0^t \dot u_\Lambda(s){\rm d}s,$ there exists a constant $c=c(c_H,T)$ with
\begin{equation}
\int_0^T \Vert u_\Lambda(s)\Vert_{H}^2{\rm d}s\leq c\big[\Vert u_0\Vert^2_V+\Vert \dot u_\Lambda\Vert_{L^2(0,T;H)}^2\big],
\end{equation}
where $c_H$ is the embedding constant of the embedding of $V$ into $H.$
This estimate and (\ref{estimation de la derivee approche}) yield the estimate
\begin{equation}
\Vert u_\Lambda\Vert_{H^1(0,T;H)}^2\leq c\big[\Vert u_0\Vert^2_V+\Vert f\Vert_{L^2(0,T;H)}^2\big]
\end{equation}
for some constant $c=c(\alpha, M, c_H, g(T), T)>0$ independent of the subdivision $\Lambda.$
\par\noindent e)  It follows from the parts $a)-d)$ that $u_\Lambda$ is bounded in $H^1(0,T;H).$ On other hand and as mentioned, Problem  (\ref{nCP in V'}) has a unique solution $u$ in $MR(V,V')$ and  we have seen in Proposition \ref{convergence forte de solution approchee} that $MR(V,H)\ni u_\Lambda\rightarrow u$ in $MR(V,V').$   As a consequence $u\in MR(V,H).$ This completes the proof.
\end{proof}



\begin{thebibliography}{999}
\bibitem{Ar06}  W.\ Arendt. \textit{Heat kernels.} $9^{th}$ Internet Seminar (ISEM) 2005/2006. Available at. https://www.uni-ulm.de/mawi/iaa/members/professors/arendt.html

\bibitem{ABHN11}  W.\ Arendt, C.J.K.\ Batty and M.\ Hieber.\ F.\ Neubrander. \textit{Vector-valued Laplace Transforms and Cauchy Problems.} Birk\"auser Verlag,  Basel, 2011.


\bibitem{AC10} {\it W.\ Arendt and R.\ Chill.} Global existence for quasilinear diffusion equations in isotropic nondivergence form.
Ann.\ Scuola Norm.\ Sup.\ Pisa CI.\ Sci.\ (5) Vol.\ IX (2010), 523-539. Zbl 1223.35202, MR2722654.
\bibitem{ADO12} W.\ Arendt, D.\ Dier and E. M.\ Ouhabaz. \textit{Invariance of convex sets for non-autonomous evolution equations governed by forms.} Available at http://arxiv.org/abs/1303.1167.
\bibitem{ADLO14} W. Arendt, D. Dier, H.~Laasri and E. M.\ Ouhabaz. \textit{Maximal regularity for evolution equations governed by non-autonomous forms},  Adv.Differential Equations 19 (2014), no. 11-12, 1043-1066.
\bibitem{AM} W.\ Arendt, S.\ Monniaux. Maximal regularity for non-autonomous Robin boundary conditions. 2014. Available at http://arxiv.org/abs/1410.3063v1
\bibitem{AuJaLa14} B.\ Augner, B.\ Jacob and H.\ Laasri \textit{On the right multiplicative perturbation of non-autonomous $L^p$-maximal regularity.} J. Operator Theory 74:2(2015), 391-415.
\bibitem{Bar71} C.\ Bardos. A regularity theorem for parabolic equations. \textit{J. Functional Analysis,\ } 7 (1971), 311-322.
\bibitem{Bre11} H.\ Br\'ezis. \emph{Functional Analysis, Sobolev Spaces and Partial Differential Equations}.
Springer, Berlin 2011.
\bibitem{D1} D.\ Dier. {Non-Autonomous Maximal Regularity for Forms of Bounded Variation.}  J. Math. Anal. Appl. 425 (2015), no. 1 33-54.
\bibitem{D2} D.\ Dier. \textit{Non-autonomous evolutionary problems governed by forms: maximal regularity and
invariance.} PhD-Thesis, Ulm, 2014.
\bibitem{DL88} R.\ Dautray and J.L.\  Lions. \textit{Analyse Math\'ematique et
Calcul Num\'erique pour les Sciences et les Techniques.} Vol.\ 8,
  Masson, Paris, 1988.
\bibitem{ELKELA11}  O.~El-Mennaoui, V.~Keyantuo, H.~Laasri.
 {\it Infinitesimal product of semigroups. Ulmer Seminare.} Heft
16 (2011), 219--230.
\bibitem{ELLA13} O. ~El-Mennaoui, H.~Laasri. {\it Stability for non-autonomous linear evolution equations with $L^p-$ maximal regularity.} \textit{ Czechoslovak Mathematical Journal.} 63 (138) 2013.
\bibitem{LASA14} H. Laasri, A. Sani. {\it Evolution Equations governed by Lipschitz Continuous Non-autonomous Forms.} Czechoslovak Mathematical Journal 65 (140) 2015.
 \bibitem{Ka} T. Kato. \textit{Perturbation theory for linear operators.} Springer-Verlag, Berlin 1992.
\bibitem{LH} H. Laasri. {\it Probl\`emes d'\'evolution et int\'egrales produits dans les espaces de Banach}. Th\`e se de Doctorat, Facult\'e des science Agadir 2012.
\bibitem{Lio61} J.L.\ Lions. \textit{ Equations diff\'erentielles op\'erationnelles
 et probl\`emes aux limites.} Springer-Verlag, Berlin, G\"ottingen, Heidelberg, 1961.
 \bibitem{OH14} B.~Haak, E. M. ~Ouhabaz. \textit{Maximal regulariry for non-autonomous evolution equations.}
\bibitem{O15} E. M.\ Ouhabaz. \textit{Maximal regularity for non-autonomous evolution equations governed by forms having less regularity.} Arch. Math. 105 (2015), 79-91.
 
\bibitem{OS10} E. M.\ Ouhabaz and C.\ Spina. \textit{Maximal regularity
for nonautonomous Schr\"odinger type equations.}  J.\ Differential
Equation 248 (2010),1668-1683.
\bibitem{Sho97} R.\ E.\ Showalter.\textit{ Monotone Operators in Banach
Space and Nonlinear Partial Differential Equations.} Mathematical
Surveys and Monographs.  American Mathematical Society, Providence,
RI, 1997.
\bibitem{Tan79} H. Tanabe. \textit{Equations of Evolution.} Pitman 1979.
\bibitem{Tho03} S. Thomaschewski. \textit{Form Methods for Autonomous and Non-Au\-ton\-o\-mous Cauchy Problems},
PhD Thesis, Universit\"at Ulm 2003.

\end{thebibliography}
\end{document}